\documentclass{amsart}
\usepackage[usenames]{color}
\usepackage{amsmath, amssymb, latexsym, multicol, color, enumerate,comment,booktabs}
\usepackage{amsthm}
\usepackage{float}
\usepackage{tikz,url}
\usepackage{graphicx, caption}
\usepackage{subcaption}
\usepackage[all]{xy}

\newtheorem{thm}{Theorem}

\newtheorem{lem}[thm]{Lemma}

\newtheorem{definition}[thm]{Definition}
\newtheorem{example}[thm]{Example}

\newtheorem{cor}[thm]{Corollary}
\newtheorem{rmk}[thm]{Remark}

\newcommand{\NN}{\mathbb{N}}
\newcommand{\ZZ}{\mathbb{Z}}
\newcommand{\QQ}{\mathbb{Q}}
\newcommand{\RR}{\mathbb{R}}

\newcommand{\CC}{\mathbb{C}}
\newcommand{\cO}{\mathcal{O}}

\title{Imaginary Multiquadratic Fields of Class Number 1}
\author{Amy Feaver}
\address{Amy Feaver, Department of Mathematics and Computing Science, The King's University, 9125 50 St NW, Edmonton, AB T6B 2H3}
\email{amy.feaver@kingsu.ca}

\keywords{class number, multiquadratic number field}

\subjclass[2010]{Primary 11R04, 11R29}

\begin{document}

\begin{abstract}
In this paper, we present a complete classification of all imaginary $n$-quadratic fields with class number 1.
\end{abstract}

\maketitle

\section{Introduction}

An \textit{$n$-quadratic number field}, $n\geq0$, is any field $K$ of degree $2^n$ over $\QQ$ which is formed by adjoining the square root of $m$ rational integers to $\QQ$ for some $m\in\NN$. That is, $K=\QQ(\sqrt{a_1},...,\sqrt{a_m})$ for $a_1,...,a_m\in\ZZ$. If $n\geq2$, the field is also called a \textit{multiquadratic} number field.  A multiquadratic field is said to be \textit{imaginary} if any of the radicands $a_1,...,a_m$ are negative.

The purpose of this paper is to provide a complete list of the imaginary $n$-quadratic fields, $n\geq1$, of class number 1.  We first look at the development and previous results for this problem in the cases of $n=1$ and $n=2$. We then go on to prove results for all imaginary $n$-quadratic fields with $n\geq3$. 

One of the first known discussions of class numbers of number fields can be found in section V of Gauss's book \textit{Disquisitiones Arithmeticae} ~\cite{Gau} which was published in 1801. Using the language of quadratic forms, Gauss made several significant conjectures and proved some results about the class numbers of quadratic fields. 

Gauss showed that the class number of $\QQ(\sqrt{a})$, denoted $h(a)$, is equal to 1 for $a\in\{-1,-2,-3,-7,-11,-19,-43,-67,-163\}$. He conjectured that his list was complete. This conjecture was proven to be true by Stark in 1967 ~\cite{Sta67}. Gauss also conjectured that $h(a)\to\infty$ as $a\to-\infty$, and this was proven by Heilbronn in 1934 ~\cite{Hei34}. Heilbronn used an unusual technique: he first showing the conjecture holds if the generalized Riemann hypothesis (GRH) is true, and then showed it holds assuming the GRH is false.

Gauss also conjectured that there are infinitely many real quadratic fields of class number 1. This is still an open problem, though many mathematicians have published results which support this conjecture. Some of the most famous work on this problem was done by Cohen and Lenstra in 1983. These results are a series of more specific conjectures on the class numbers of quadratic fields, called the Cohen-Lenstra Heuristics, and are well supported by numerical data. In particular, they conjecture that real quadratic fields with prime discriminant have class number 1 close to 75.446\% of the time. For a more in-depth discussion see chapter 5, section 10 of Cohen's book ~\cite{Coh}.

While the class numbers of real quadratic fields still remain fairly elusive, we do have the tools to study the class numbers of imaginary multiquadratic fields more easily. In 1974, Brown and Parry determined a complete list of imaginary biquadratic number fields with class number 1~\cite{BP74}. There are exactly 47 such fields. These are the fields $\QQ\left(\sqrt{a_1},\sqrt{a_2}\right)$ with the sets $\{a_1,a_2\}$ as follows

\vspace{1em} 

\begin{center}

\begin{tabular}{lllll}

\{-1,2\}  & \{2,-3\}   & \{-3,5\}    & \{-7,5\}    & \{-11,17\}   \\
\{-1,3\}  & \{2,-11\}  & \{-3,-7\}   & \{-7,-11\}  & \{-11,-19\}  \\
\{-1,5\}  & \{-2,-3\}  & \{-3,-11\}  & \{-7,13\}   & \{-11,-67\}  \\
\{-1,7\}  & \{-2,5\}   & \{-3,17\}   & \{-7,-19\}  & \{-11,-163\} \\
\{-1,11\} & \{-2,-7\}  & \{-3,-19\}  & \{-7,-43\}  & \{-19,-67\}  \\
\{-1,13\} & \{-2,-11\} & \{-3,41\}   & \{-7,61\}   & \{-19,-163\} \\
\{-1,19\} & \{-2,-19\} & \{-3,-43\}  & \{-7,-163\} & \{-43,-67\}  \\
\{-1,37\} & \{-2,29\}  & \{-3,-67\}  &           & \{-43,-163\} \\
\{-1,43\} & \{-2,-43\} & \{-3,89\}   &           & \{-67,-163\} \\
\{-1,67\} & \{-2,-67\} & \{-3,-163\} &           &            \\
\{-1,163\}&          &  				 &           &            \\

\end{tabular}

\end{center}

\vspace{1em} 

Brown and Parry determined this list using techniques of Herglotz~\cite{Her22} who gives a formula that relates the class number of a biquadratic number field to that of its quadratic subfields, along with some other parameters. This should not come as a surprise to the keen observer who has surely already noticed that many of the numbers in the table above are also radicands of quadratic fields with class number 1. This relationship between the class number of a field and its quadratic subfields is what drives the discussion in the next section.

The main result of this paper is stated and proven in Section \ref{main_theorem_section}. Here we prove that there are no $n$-quadratic fields with $n\geq4$ that have class number 1. In the case of $n\geq5$, this result follows from a theorem of Fr\"{o}lich, stated as Theorem \ref{ramPrimes}. Much of the machinery for the classification in the cases of $n=3$ and $n=4$ is based on Kuroda's class number formula, which was developed by Lemmermeyer in 1994 ~\cite{Lem94}. The applications of this formula that we need for this classification are found in section \ref{Kuroda_section}.

The complete list of fields of class number 1 in the $n=3$ case. These are exactly the fields $\QQ\left(\sqrt{a_1},\sqrt{a_2},\sqrt{a_3}\right)$ with the sets $\{a_1,a_2,a_3\}$ as follows:

\vspace{1em}

\begin{center}
\begin{tabular}{lllll}
$\{-1,2,3\}$  & $\{-1,3,5\}$  & $\{-1,7,5\}$  & $\{-2,-3,-7\}$  & $\{-3,-7,5\}$  \\
$\{-1,2,5\}$  & $\{-1,3,7\}$  & $\{-1,7,13\}$ & $\{-2,-3,5\}$ & $\{-3,-11,2\}$  \\
$\{-1,2,11\}$ & $\{-1,3,11\}$ & $\{-1,7,19\}$ & $\{-2,-7,5\}$ & $\{-3,-11,-19\}$ \\
           & $\{-1,3,19\}$ &            &            & $\{-3,-11,17\}$ \\
\end{tabular}.
\end{center}

\vspace{1em}

In addition to considering the class number 1 problem, there are several other lists of quadratic and biquadratic fields of fixed class number. The problem of finding a complete list of imaginary quadratic fields $\QQ\left(\sqrt{-a}\right)$, $a>0$, with class number equal to a positive integer $m$ is often referred to as \textit{Gauss' class number $m$ problem}.  

Gauss' class number two problem was solved independently by both Baker ~\cite{Bak66}, ~\cite{Bak71} and Stark ~\cite{Sta75}, between 1966 and 1975. Oesterl\'{e} solved the class number 3 problem in 1983 ~\cite{Oes83}. Arno completed the classification for class number 4 in 1992 ~\cite{Arn92}, and Wagner completed the class number 5, 6 and 7 classifications in 1996 ~\cite{Wag96}. Then, in 1998, Arno, Robinson and Wheeler solved Gauss' class number $m$ problem for odd values of $m$ with $9\leq m\leq23$ ~\cite{ARW}.

 There are fewer such lists for imaginary biquadratic fields of fixed class number. The class number 2 problem for imaginary biquadratic fields was solved by Buell, H.C. Williams and K.S. Williams ~\cite{BWW} in 1977, and the class number 3 problem by Jung, S.W., Kwon, S.H., ~\cite{JK98} in 1998. It is worth noting that the latter result also relied heavily on the application of Kuroda's class number formula.

\section{$n$-quadratic number fields}\label{multiquad_section}

In this section, we will define objects related to $n$-quadratic number fields and discuss some of their more elementary properties. 

Of course, it is redundant to adjoin more than $n$ square roots to $\QQ$ when representing an $n$-quadratic field or to use radicands that are not squarefree.  We will frequently use the abbreviation \textit{sf} to refer to the squarefree part of an integer; i.e. $s.f(20) = 5$. We want to describe these fields as cleanly as possible, so we define the following terms:

\begin{definition} Let $m,n\in\ZZ$ with $1\leq n\leq m$. A list of squarefree rational integers $\{a_1,...,a_m\}$ with $a_i\neq0,1$, $i\in\{1,...,m\}$ is called a \textit{radicand list} for the $n$-quadratic field $\QQ(\sqrt{a_1},...,\sqrt{a_m})$. Further, the list $\{a_1,...,a_m\}$ is called a \textit{primitive radicand list} if $m=n$.  
\end{definition}

Note that the radicand list $\{a_1,...,a_n\}$ is primitive if and only if the following condition holds: for any proper subset $I\subset\{1,...,n\}$, and any $j\in\{1,...,n\}\setminus I$, $a_j$ is not equal to the squarefree part of the product $\prod_{i\in I}a_i$.

We will primarily use the notation of a radicand list when discussing multiquadratic fields to avoid tedium. It is further useful to write these lists in a more canonical way to use as little notation as possible. We use the following definitions and supporting lemmas:

\begin{definition} For any rational prime $p$, a primitive radicand list $\{a_1,...,a_n\}$ of an $n$-quadratic field is said to be \textit{$p$-headed} if $p\nmid a_i$ for any $i\in\{2,...,n\}$.
\end{definition}

Note that a radicand list can be $p$-headed whether or not $p\mid a_1$, as long as $p$ does not divide the other radicands in the list. For example, the radicand list $\{6,3,5\}$ of a triquadratic field is 2-headed and also 101-headed. This list only cannot be described as 3-headed or 5-headed (nor 6-headed, as 6 is not prime). 

\begin{lem}\label{p_headed_lem} For any $n$-quadratic field $K$ and any rational prime $p$ there exists a $p$-headed radicand list $\{a_1,...,a_n\}$ for $K$.
\end{lem}

\begin{proof} Choose any prime $p$ and assume that $\{a_1',...,a_n'\}$ is a primitive radicand list for $K$ which is not $p$-headed. Then there exists $i>1$ such that $p\mid a_i'$; let $a_1=a_i'$. For each $j\in\{2,...,i-1,i+1,...,n\}$ define
\[a_j:=\left\{\begin{array}{ll}
a_j' & p\nmid a_j' \\
\frac{a_i'a_j'}{\gcd(a_i',a_j')^2} & p\mid a_j'
\end{array}\right.
\]
and similarly let
\[a_i:=\left\{\begin{array}{ll}
a_1' & p\nmid a_1'\\
\frac{a_i'a_1'}{\gcd(a_i',a_1')^2} & p\mid a_1'
\end{array}.\right.
\]
Thus we have a set $\{a_1,...,a_n\}$ such that $p\mid a_1$. Also, it is easy to see that each $a_k$, $1\leq k\leq n$ is squarefree and not equal to 1. Therefore, $\{a_1,...,a_n\}$ satisfies the conditions of a primitive set of generators and is $p$-headed.
\end{proof}

\begin{definition} A primitive radicand list $\{a_1,...,a_n\}$ is said to be in \textit{standard form} if 
\begin{enumerate}
\item it is 2-headed, and
\item for any $i,j\in\{1,...,n\}$ with $2\nmid a_ia_j$, we have that $a_i\equiv a_j\bmod 4$.
\end{enumerate}
\end{definition}

Note that if $K$ is an $n$-quadratic field then there are usually multiple radicand lists for $K$ in standard form.

\begin{lem} For any $n$-quadratic field $K$ there exists a primitive radicand list $\{a_1,...,a_n\}$ for $K$ written in standard form.
\end{lem}  

\begin{proof} Let $\{a_1',...,a_n'\}$ be a primitive radicand list for $K$, and, by Lemma \ref{p_headed_lem}, we may assume that this list is 2-headed. If this set is not written in standard form, then there exists an $i\in\{1,...,n\}$ such that $a_i'\equiv3\bmod4$, and there also exists a nonempty subset of $A$ of $\{a_1',...,a_n'\}$ such that $a\equiv1\bmod4$ for all $a\in A$.

For each $j\in\{1,...,n\}$ define 
\[a_j:=\left\{\begin{array}{ll}
a_j' & a_j'\equiv2,3\pmod 4 \\
\frac{a_i'a_j'}{\gcd(a_i',a_j')^2} & a_j'\equiv1\bmod4
\end{array}.\right.\]
Then $\{a_1,...,a_n\}$ is a radicand list for $K$ in standard form.
\end{proof}

Some examples of the above definitions are provided in the following table:

\vspace{0.5em}

\begin{center}
\begin{tabular}{|c|c|c|c|}
\hline
4-quadratic  & a primitive  &  a 3-headed & a radicand list for $K$\\
number field $K$ & radicand list for $K$ & radicand list for $K$&  in standard form \\
\hline
$\QQ\left(\sqrt{2},\sqrt{6},\sqrt{7},\sqrt{3},\sqrt{13}\right)$ & $\{2,6,7,13\}$ & $\{6,2,7,13\}$ & $\{2,3,7,39\}$\\
\hline
$\QQ\left(\sqrt{-17},\sqrt{20},\sqrt{7},\sqrt{-1}\right)$ & $\{-17,5,7,-1\}$ & $\{-17,5,7,-1\}$ & $\{-17,-5,7,-1\}$\\
\hline
$\QQ\left(\sqrt{-3},\sqrt{5},\sqrt{-7},\sqrt{17}\right)$ & $\{-3,5,-7,17\}$ & $\{-3,5,-7,17\}$ & $\{-3,5,-7,17\}$\\
\hline
\end{tabular}
\end{center}

\vspace{1em}

In addition to using these standards when writing radicand lists, we will find that in the case of imaginary $n$-quadratic fields it may be useful to write the radicand list in another way, where each radicand is a negative integer. The existence (and a more formal description) of such a list is established below:

\begin{lem} Let $K$ be an imaginary $n$-quadratic field. Then there exist positive squarefree integers $a_1,...,a_n$ such that $\{-a_1,...,-a_n\}$ is a primitive radicand list for $K$.
\end{lem}

\begin{proof} Let $\{a_1',...,a_n'\}$ be a primitive radicand list for $K$. Since $K$ is not totally real, at least one element of this radicand list must be negative. Thus, without loss of generality assume this list is ordered so that for some $i$, $1<i\leq n$ we have that $a_1',...,a_i'<0$ and $a_{i+1}',...,a_n'>0$. Let $a_j=|a_j'|$ for all $j\leq i$. Also, for $j$ satisfying $i<j\leq n$ set $a_j=|sf(a_1'a_j')|$. Then $\{-a_1,...,-a_n\}$ is a primitive radicand list for $K$ satisfying the lemma.
\end{proof}

\subsection{Subfields}

The Galois group of a multiquadratic field is a well-known abelian group: if $K$ is an $n$-quadratic field then the extension $K/\QQ$ is Galois, with Galois group $G:=\text{Gal}(K/\QQ)\cong\bigoplus_{i=1}^n(\ZZ/2\ZZ)$. Thus, we immediately know the number of subfields of $K$, along with their degrees. At times, we can glean information about a multiquadratic field simply from its quadratic subfields. Thus we begin by observing that any $n$-quadratic field with $n\geq1$ has $2^n-1$ quadratic subfields.

To further describe these quadratic subfields, we need the following definition. Its usefulness will become clear by the lemma immediately following it.

\begin{definition} Let $m,n\in\ZZ$ with $1\leq n\leq m$. A list of squarefree rational integers $\{a_1,...,a_m\}$ is called a \textit{complete radicand list} for an $n$-quadratic field if and only if the following conditions hold:
\begin{enumerate}
\item $m=2^n-1$ and
\item $a_i\neq a_j$ for all $i,j\in\{1,...,m\}$ with $i\neq j$.
\end{enumerate}
\end{definition}

Note that the complete radicand list for a number field is much larger than the primitive radicand list. For example the $n$-quadratic field with primitive radicand list $\{-1,2,3,5\}$ has complete radicand list 
\[\{-1,-2,2,-3,3,-5,5,-6,6,-10,10,-15,15,-30,30\}.\]

\begin{lem} If $K$ is any $n$-quadratic field, then the set of squarefree integers, $\{a_1,...,a_{2^n-1}\}$ is a complete radicand list for $K$ if and only if the fields given by $\QQ(\sqrt{a_i})$, $1\leq i\leq2^n-1$, are exactly the $2^n-1$ distinct quadratic subfields of $K$.
\end{lem}

It is easy to see that, given a primitive radicand list $\{a_1,...,a_n\}$ for an $n$-quadratic field $K$, we may construct a complete radicand list for $K$, and that the set of integers in this list is unique. The list is exactly 
\[\left\{sf\left(\prod_{i\in I}a_i\right)\text{ such that } \emptyset\subset I\subseteq\{1,...,n\}\right\}.\]

Now that we can determine a complete list of quadratic subfields of any $n$-quadratic field, we consider the real and imaginary subfields when negative radicands are present:

\begin{lem}\label{numberOfQuadraticSubfields} Let $K$ be an imaginary $n$-quadratic number field with $n>1$. Then there are $2^{n-1}$ imaginary quadratic subfields and $2^{n-1}-1$ real quadratic subfields of $K$.
\end{lem}

\begin{proof} Since $K$ is an imaginary $n$-quadratic field there exist positive integers $a_1,...,a_n$ such that $\{-a_1,...,-a_n\}$ is a primitive radicand list for $K$. 

We begin by finding an upper bound on the number of imaginary quadratic subfields of $K$: $K$ contains the $n$ imaginary quadratic fields given by $\QQ(\sqrt{-a_i})$, $1\leq i\leq n$. Also, if $n\geq3$, we see that $K$ will also contain the $n\choose 3$ imaginary quadratic fields given by $\QQ(\sqrt{-a_ia_ja_k})$, $1\leq i,j,k\leq n$, and $i,j,k$ all distinct. Similarly, any product $-\pi$ of an odd number $\ell$ of distinct $-a_i$'s will give rise to an imaginary quadratic subfield $\QQ(\sqrt{-\pi})$. There will be $n\choose\ell$ such subfields, and other than these constructions of imaginary quadratic subfields of $K$, 
 other possibilities. Therefore, the number of imaginary quadratic subfields of $K$ is at most
\[\#Im\leq {n\choose 1}+{n\choose 3}+{n\choose 5}+\cdots+{n\choose {2\left\lfloor \frac{n-1}{2}\right\rfloor}+1}=2^{n-1}.\]
Now let's consider the number of real quadratic subfields of $K$. Using an argument similar to the one above, we construct all possible real quadratic fields by taking products of an even number of the $-a_i$'s, $1\leq i\leq n$. This gives an upper bound on the number of real quadratic subfields of $K$:
\[\#Re\leq {n\choose 2}+{n\choose 4}+{n\choose 6}+\cdots+{n\choose {2\left\lfloor \frac{n-1}{2}\right\rfloor}}=2^{n-1}-1.\]
Thus  $\#Im+\#Re\leq 2^{n-1}+2^{n-1}-1=2^n-1$. But this upper bound is equal to the total number of quadratic subfields of $K$. Therefore these upper bounds are exactly equal to the number of these fields, so the lemma is established.
\end{proof}

\subsection{Ramification of primes in $n$-quadratic fields} 

In the case of quadratic fields $\QQ(\sqrt{a})$, with $a\in\ZZ$ squarefree, we know exactly when a rational prime ramifies. For any prime $p$, if $p|a$ then $p$ ramifies and its factorization is
\[p\cO_{\QQ(\sqrt{a})}=(p,\sqrt{a})^2.\]
Additionally, if $a\equiv3\bmod4$ then $2$ ramifies;
\[2\cO_{\QQ(\sqrt{a})}=(2,1+\sqrt{a})^2.\]

Considering this information, one would probably conjecture correctly which rational primes ramify in an $n$-quadratic field, $n\geq2$, with primitive radicand list $\{a_1,a_2,...,a_n\}$: those primes which divide any of the radicands $a_i$, and 2 will also ramify if any of the $a_i$'s are congruent to 3 mod 4. Let's look at why this is the case, and use our discussion to further explore important properties of the ramification of primes.

Let $\Delta_K\in\ZZ$ be the discriminant of a number field $K$. It is standard fact that $p\in\ZZ$ ramifies in $\cO_K$ if and only if $p|\Delta_K$. Thus we will approach this discussion by first presenting the following theorem.
\begin{thm} [ {Schmal \cite[Theorem 2.1]{Sch}} ] Let $K$ be an $n$-quadratic field, $n\geq2$, with radicand list $\{a_1,a_2,...,a_n\}$ written in standard form. Let $\prod_{j=1}^sp_j^{m_j}$ be the prime factorization in $\ZZ$ of the product of the radicands, $\prod_{i=1}^na_i$. Then
\[\Delta_K=(2^ep_1\cdots p_s)^{2^{n-1}},\]
where
\[ e= \left\{\begin{array}{ll}
0 & \text{ if }a_1\equiv1\bmod4\\
2 & \text{ if }(a_1,a_2)\equiv(2,1)\text{ or }(3,1)\bmod4 \\
3 & \text{ if }(a_1,a_2)\equiv(2,3)\bmod4\end{array}\right.\]
\end{thm}

Therefore:

\begin{cor}\label{inertiaFields} Let $K$ be an $n$-quadratic field with radicand list $\{a_1,a_2,...,a_n\}$ written in standard form. For any prime $p\in\ZZ$ the following describes the inertia field of $K$ with respect to the prime $p$ and the ramification index of $p$ in $\cO_K$:

First, when $p=2$ we have: 
\begin{enumerate}
\item If $a_1\equiv a_2\equiv\cdots\equiv a_n\equiv 1\pmod4$ then $(2)$ is unramified in $\cO_K$ and the inertia field is $K$.
\item If $a_1\not\equiv1\pmod4$ and $a_i\equiv1\pmod4$, $2\leq i\leq n$, then $2\cO_K$ has ramification index 2 and the inertia field is the $(n-1)$-quadratic field $\QQ(\sqrt{a_2},...,\sqrt{a_n})$.
\item If $a_1\equiv2\pmod4$ and $a_i\equiv3\pmod4$, $2\leq i\leq n$, then $2\cO_K$ has ramification index 4. The inertia field is the $(n-2)$-quadratic field $\QQ(\sqrt{a_2a_3},...,\sqrt{a_2a_n})$.
\end{enumerate}

When $p$ is an odd prime we have:
\begin{enumerate}
\item If $p\nmid a_i$ for all $i\in\{1,...,n\}$ then $(p)$ is unramified in $\cO_K$ and the inertia field is $K$.
\item Otherwise $p\cO_K$ has ramification index 2. Let $\{a_1',...,a_n'\}$ be a $p$-headed radicand list for $K$, then the $(n-1)$-quadratic field $\QQ(\sqrt{a_2'},...,\sqrt{a_n'})$ is the inertia field of $K$ for $p$.
\end{enumerate}
\end{cor}

From these facts about the ramification of primes in multiquadratic fields, we get the following lemma:

\begin{lem}\label{numberOfRamifiedPrimes} Let $K$ be an $n$-quadratic field. Then:
\begin{enumerate} 
\item If $K$ is totally real, the number of primes ramified in $K$ is at least $n$.
\item If $K$ is imaginary, the number of primes ramified in $K$ is at least $n-1$.
\end{enumerate}
\end{lem}

\begin{proof} First note that if we can prove (1) then (2) immediately follows. That is, if $K$ is an imaginary $n$-quadratic field, then $K$ has a real $(n-1)$-quadratic subfield. Assuming that (1) holds it follows that there are at least $n-1$ primes ramified in this subfield and thus in $K$.

Thus it only remains to prove this lemma for real $n$-quadratic fields. Let $K$ be a real $n$-quadratic field with primitive radicand list $\{a_1,...,a_n\}$, $a_i>1$ squarefree for $1\leq i\leq n$. Then it follows that every prime which divides the product $\pi_K=\prod_{i=1}^na_i$ ramifies in $K$. Therefore, it is sufficient to prove that at least $n$ primes divide the product $\pi_K$. 

We will prove this by induction on $n$. First if $\QQ(\sqrt{a_1})$ is a real quadratic field then clearly at least one prime divides the product $\pi_{\QQ(\sqrt{a_1})}=a_1$. Assume that for any real $(n-1)$-quadratic field $L$, there are at least $(n-1)$ primes dividing the product of the radicands $\pi_L$. Now consider the real $n$-quadratic field $K$. Then we can clearly see that at least one prime $p$ divides the product $\pi_K$. Thus there exists a $p$-headed radicand list $\{a_1',a_2',...,a_n'\}$ such that $p|a_1'$ but $p\nmid a_i'$ for $2\leq i\leq n$. Then there is an $(n-1)$-quadratic subfield $K'$ of $K$ with radicand list $\{a_2',...,a_n'\}$. Since there are at least $n-1$ primes dividing $\pi_{K'}$ and none of these primes are equal to $p$, then there must be at least $n$ primes dividing $\pi_K$.
\end{proof}

It is worth noting that these lower bounds on the number of ramified primes are the best possible. In the real multiquadratic case, an $n$-quadratic field with radicand list $\{2,p_1,p_2,...,p_{n-1}\}$, where $p_1,p_2,...,p_{n-1}$ are distinct odd primes, has exactly n ramified primes. If we were to add $-1$ to this list, we would have the radicand list $\{-1,2,p_1,p_2,...,p_{n-1}\}$ defining an imaginary $n+1$-quadratic field. There are still exactly $n$ ramified primes in this imaginary multiquadratic field.

\begin{lem}\label{nMinus2Field} Let $K$ be an imaginary $n$-quadratic field with $n\geq3$. There exists an odd prime $p$ and a real $(n-2)$-quadratic subfield $k$ of $K$ such that $p$ is unramified in $k$ but is ramified in $K$.
\end{lem}

\begin{proof} By Lemma \ref{numberOfRamifiedPrimes} there are at least $n-1$ primes which ramify in $K$. Since $n\geq3$ at least two primes ramify in $K$, and thus at least one of these primes must be odd. Choose any odd prime which ramifies in $K$ and denote it by $p$. By Corollary \ref{inertiaFields} there exists an $(n-1)$-quadratic subfield $K_E$ of $K$ which is the inertia field for $p$. If $K_E$ is totally real we may choose $k$ to be any $(n-2)$-quadratic subfield of $K_E$. If $K_E$ is imaginary write $K_E=\QQ(\sqrt{-a_1},...,\sqrt{-a_{n-1}})$. Then we may choose $k=\QQ\left(\sqrt{sf(a_1a_2)},...,\sqrt{sf(a_1a_{n-1}})\right)$.
\end{proof}

\section{Results for class numbers of $n$-quadratic fields}\label{previous_results_section}

\subsection{Class number theorems using ramified primes}

The goal of this paper is to lay out a complete list of imaginary $n$-quadratic fields of class number 1. In this section we present two theorems, one which hits this classification with a sledgehammer, eliminating all multiquadratic fields of large enough degree from the list of potential candidates (namely, $n$-quadratic fields with $n\geq6$, and some others). Other techniques will be applied later to imaginary multiquadratic fields of bounded degree. It is still useful to think about the class numbers of $n$-quadratic fields with $n\geq6$ in terms of the latter techniques, as they still give more insight into the actual value of the class number.

This sledgehammer is the following theorem, due to Fr\"{o}lich~\cite{Fro}:

\begin{thm} [ {Fr\"{o}lich \cite[Theorem 5.6]{Fro}}]\label{ramPrimes} Let $K$ be a real Abelian field of 2-power degree. The class number of $K$ is even whenever the number of finite rational primes that ramify in $K$ is greater than or equal to 5.
\end{thm}

In light of this theorem, along with Corollary \ref{numberOfRamifiedPrimes} we can deduce:

\begin{cor} Let $h_K$ denote the class number of an $n$-quadratic field $K$. Then
\begin{enumerate}
\item if $K$ is totally real and $n\geq5$ then $2\mid h_K$ and
\item if $K$ is imaginary and $n\geq6$ then $2\mid h_K$
\end{enumerate}
\end{cor}

Mouhib studied the 2-part of the ideal class group of real 4-quadratic fields with prime radicands and came up with a related result. We will not apply this result directly in this paper, as it focuses on real multiquadratic fields rather than imaginary. However, this is interesting to any reader and relevant to those who wish to push this study further. His theorem states:

\begin{thm} [ {Mouhib \cite[Main Theorem]{Mou}} ]Let $p_1,p_2,p_3,p_4>0$ be distinct primes and $K=\QQ(\sqrt{p_1},\sqrt{p_2},\sqrt{p_3},\sqrt{p_4})$. Then the 2-class group of $K$ is trivial if and only if, after a suitable permutation of the indices, the $p_i$ have one of the following properties:
\begin{enumerate}
\item $p_1=2$, $p_2\equiv p_3\equiv p_4\equiv-1\pmod4$, $\left(\frac{2}{p_2}\right)=-\left(\frac{2}{p_3}\right)=-\left(\frac{2}{p_4}\right)=1$ and $\left(\frac{p_2}{p_3}\right)\left(\frac{p_2}{p_4}\right)=-1$
\item $p_1=2$, $p_2\equiv p_3\equiv p_4\equiv-1\pmod4$, $\left(\frac{2}{p_2}\right)=-1$, and
\begin{enumerate}[(i)]
\item $\left(\frac{p_2}{p_3}\right)=\left(\frac{p_2}{p_4}\right)=-1$ and $\left(\frac{2}{p_3}\right)\left(\frac{2}{p_4}\right)=-1$, or
\item $\left(\frac{2}{p_3}\right)=\left(\frac{2}{p_4}\right)=-1$ and $\left(\frac{p_2}{p_3}\right)\left(\frac{p_2}{p_4}\right)=-1$, or
\item $\left(\frac{p_2}{p_3}\right)\left(\frac{p_2}{p_4}\right)=\left(\frac{2}{p_3}\right)\left(\frac{2}{p_4}\right)=-1$ and $\left(\frac{p_2}{p_3}\right)\neq\left(\frac{2}{p_3}\right)$.
\end{enumerate}
\end{enumerate}
\end{thm}

\section{Kuroda's class number formula}\label{Kuroda_section}

Now we must develop techniques to study the imaginary $n$-quadratic fields with $n\in\{3,4,5\}$. Lemmermeyer~\cite{Lem94} derived the following formula, based on work previously done by Kuroda:

\begin{thm}
[{Lemmermeyer \cite[Theorem 1]{Lem94}}]\label{KCNF}
 (Kuroda's Class Number Formula) For any number field $L$ define $E(L)$ to be the unit group of $\cO_L$. Let $K/k$ be a $V_4$ extension of number fields and let $k_i$, $i\in\{1,2,3\}$, be the three number fields such that $k\subsetneq k_i\subsetneq K$. Let $h_i$ denote the class number of $k_i$, $i\in\{1,2,3\}$. Then we have
\[h_K=2^{d-\kappa-2-\nu}q(K/k)h_1h_2h_3/h_k^2,\]
where $q(K/k)$ is the (finite) unit index given by $q(K/k):=[E(K):E(k_1)E(k_2)E(k_3)],$ $d$ denotes the number of infinite places ramified in $K/k$, $\kappa$ is the $\ZZ$-rank of $E(k)$ and
\[\nu = \left\{
\begin{array}{ll}
1 & \text{if } K=k(\sqrt{\epsilon},\sqrt{\eta}),\ \epsilon,\eta\in E(k)\\
0 & \text{otherwise}
\end{array} .
\right.\]
In particular,

\[h_K = \left\{
\begin{array}{ll}
2^{-2}q(K/\QQ)h_1h_2h_3 & \text{if } K \text{ is a real biquadratic field}\\
2^{-1}q(K/\QQ)h_1h_2h_3 & \text{if } K \text{ is an imaginary biquadratic field} 
\end{array} .
\right.\]

\end{thm}

\vspace{1em}

We will apply this theorem to imaginary $n$-quadratic fields $K$ with $n\geq3$, viewed as a $V_4$ extension over a real $(n-2)$-quadratic field $k$:

\vspace{1em}

\[
\xymatrix{
 & K=\QQ\left(\sqrt{-a_1},...,\sqrt{-a_n}\right) \ar@{-}[d] \ar@{-}[dr]	& \\
k_1=k\left(\sqrt{-a_1}\right)  \ar@{-}[ur] & k_2=k\left(\sqrt{-a_2}\right) \ar@{-}[d] & k_3=k\left(\sqrt{a_1a_2}\right) \ar@{-}[dl]  \\
 & k=\QQ\left(\sqrt{a_2a_3},\sqrt{a_2a_4},...,\sqrt{a_2a_n}\right) \ar@{-}[ul]	 			
}
\]

\vspace{2em}

\begin{rmk} \normalfont
In this case, it is easy to see that the unit index $q(K/k)$ is finite. By Dirichlet's unit theorem, the unit groups $E(K)$ and $E(k_3)$ have the same rank. Since these are finitely generated abelian groups with $E(k_3)\subseteq E(K)$ we clearly have that $[E(K):E(k_1)E(k_2)E(k_3)]\in\NN$.
\end{rmk}

\begin{lem}\label{smallKuroda} Let $K$ be an imaginary $n$-quadratic field with $n\geq3$. Let $k$ be any real $(n-2)$-quadratic subfield of $K$ such that there exists an odd prime which ramifies in $K/k$. Then we have the following formula for the class number $h_K$ of $K$:
\[h_K=\frac12 q(K/k)h_1h_2h_3/h_k^2.\]
Here $h_k$ is the class number of $k$, $h_1,h_2,h_3$ are the class numbers of the three $(n-1)$-quadratic fields $k_1,k_2,k_3$ between $k$ and $K$, and $q(K/k)=[E(K):E(k_1)E(k_2)E(k_3)]$.
\end{lem}

\begin{proof} Since $K/k$ is a $V_4$ extension of number fields we may apply Kuroda's class number formula, as defined in Theorem \ref{KCNF}:
\[h_K=2^{d-\kappa-2-\nu}q(K/k)h_1h_2h_3/h_k^2.\]

Since $\kappa$ is the $\ZZ$-rank of the unit group $E(k)$, we may apply Dirichlet's unit theorem and find that $\kappa = 2^{n-2}-1$. Next, recall that $d$ is the number of infinite places ramified in $K/k$. Now there are $2^{n-2}$ distinct infinite places in $k$ since it is a real $(n-2)$-quadratic field; call these embeddings $\sigma_1,...,\sigma_n:k\hookrightarrow \RR$. Each $\sigma_i$, $i=1,...,2^{n-2}$ is going to extend to an embedding $\tau_i:K\hookrightarrow\CC$. Now $\tau_i$ is complex embedding since $K$ is an imaginary $n$-quadratic field, so if $\tau_i$ lies over $\sigma_i$, then $\overline{\tau_i}$ also lies over $\sigma_i$. We know $\tau_i$ and $\overline{\tau_i}$ correspond to the same infinite place in $K$, so $\sigma_i$ must be ramified for all $i\in\{1,...,2^{n-2}\}$. Therefore, there are $2^{n-2}$ infinite places which ramify in $K$, forcing $d=2^{n-2}$.

Plugging these quantities into the formula we have
\[h_K=\frac12 2^{-\nu} q(K/k)h_1h_2h_3/h_k^2.\]

Thus we wish to show $\nu=0$ to get the desired result. Denote by $p$ an odd prime which ramifies in $K$ but not in $k$; such a prime always exists for an appropriate choice of $k$ by Lemma \ref{nMinus2Field}. There exists a $p$-headed radicand list $\{-a_1,...,-a_n\}$ for $K$, with $a_1,...,a_n$ positive. 

Recall that $\nu=0$ if and only if $K$ cannot be written as $k(\sqrt{\eta},\sqrt{\epsilon})$ for $\eta,\epsilon$ in $E(k)$, the unit group of $k$. I will show $\nu=0$ by contradiction, so assume there exist units $\eta,\epsilon\in k$ such that $K=k(\sqrt{\eta},\sqrt{\epsilon})$. Then the element $\sqrt{a_1}\in K$ can be written 
\[\sqrt{-a_1}=w+x\sqrt{\eta}+y\sqrt{\epsilon}+z\sqrt{\epsilon\eta},\ w,x,y,z\in k.\]

However, since $\sqrt{-a_1}$ is purely imaginary and does not have a real part, we must have $w=0$. Also, without loss of generality, since at most 2 of $\sqrt{\epsilon}$, $\sqrt{\eta}$ and $\sqrt{\epsilon\eta}$ can be imaginary, assume $z\sqrt{\epsilon\eta}$ is real and thus $z=0$ as well. Therefore
\[\sqrt{-a_1}=x\sqrt{\eta}+y\sqrt{\epsilon},\ x,y\in k.\]
There are two cases to consider: first, the case where both $x$ and $y$ are nonzero, and second, when exactly one of $x$ or $y$ is zero. If $x,y\neq0$, then squaring both sides gives
\[-a_1=x^2\eta+y^2\epsilon+2xy\sqrt{\eta\epsilon}.\]

Now since $-a_1,x^2\eta,y^2\epsilon\in k$ and 
\[-a_1-x^2\eta-y^2\epsilon=2xy\sqrt{\eta\epsilon}\]
then we must have $\sqrt{\eta\epsilon}\in k$, so
\[K=k(\sqrt{\eta},\sqrt{\epsilon})=k(\sqrt{\eta},\sqrt{\eta\epsilon})=k(\sqrt{\eta}).\]
Thus $K$ is a degree 2 extension of $k$; which is a contradiction, because we know $K/k$ is a $V_4$ extension.

Thus assume that exactly one of $x,y$ is zero; without loss of generality assume $y=0$. Then $\sqrt{-a_1}=x\sqrt{\epsilon}$, so 
\[-a_1=x^2\epsilon.\]
This implies that the ideal $(a_1)$ is equal to a square in $\cO_k$ which is a contradiction since there is an odd prime $p|a_1$ which does not ramify in $k$. Thus $\nu=0$.

\end{proof}

\begin{thm}\label{bigKuroda} Keeping the notation above, we have
\[h_K=\left(\frac12\right)^{2^{n-1}-1}QP h_3,\]
where $P$ is the product of the class numbers of all imaginary quadratic subfields of $K$ and $Q\in\ZZ_{>0}$ is a product of unit indices $q(L/\ell)$ where $L/\ell$ are $V_4$ extensions with $L\subseteq K$.
\end{thm}

\begin{proof} Proof by induction. Base case: when $n=2$ it is easy to check that this theorem holds. Now assume that the statement holds for imaginary $(n-1)$-quadratic fields. Consider, in particular, the $(n-1)$-quadratic fields $k_1$ and $k_2$. Let $k'$ be an $(n-3)$-quadratic field such that there exists a prime $p$ which ramifies in $k$ but not in $k'$. Then $k_1/k'$ and $k_2/k'$ are $V_4$ extensions which have $k$ as a totally real intermediate field.

Then
\[h(k_1)=\left(\frac12\right)^{2^{(n-1)-1}-1} Q_1 P_1 h_k,\]
\[h(k_2)=\left(\frac12\right)^{2^{(n-1)-1}-1} Q_2 P_2 h_k.\]
Here, $P_i$, $i=1,2$ is the product of the class numbers of all imaginary quadratic subfields of $k_i$. Also, $Q_i\in\ZZ_{>0}$, $i=1,2$ are products of unit indices $q(L/\ell)$ where $L/\ell$ are $V_4$ extensions with $L\subseteq k_i\subset k$.

Then, using the fact that
\[h_K=\frac12 q(K/k)h_1h_2h_3/h_k^2,\]
and plugging in for $h_1$ and $h_2$, we find that
\begin{align*}
h_K &= \frac12 q(K/k) \left(\left(\frac12\right)^{2^{n-2}-1} Q_1 P_1 h_k\right)\left(\left(\frac12\right)^{2^{n-2}-1} Q_2 P_2 h_k\right)  h_3/h_k^2 \\
&= \left(\frac12\right)^{1+2^{n-2}-1+2^{n-2}-1} q(K/k)  Q_1 P_1 Q_2 P_2  h_3 \\
&= \left(\frac12\right)^{2^{n-1}-1} (q(K/k)Q_1Q_2) (P_1 P_2)  h_3.
\end{align*}
Setting $Q=q(K/k)Q_1Q_2$ we have that $Q\in\ZZ_{>0}$ and is a product of unit indices $q(L/\ell)$ where $L/\ell$ are $V_4$ extensions with $L\subseteq K$. Also, setting $P=P_1P_2$  we have that $P$ is a product of imaginary quadratic subfields of $K$. From Lemma \ref{numberOfQuadraticSubfields} we know that $k_1$ and $k_2$ each have $2^{n-2}$ imaginary quadratic subfields and $K$ has $2^{n-1}$ imaginary quadratic subfields. Further, we see that the set of imaginary quadratic subfields of $k_1$ and $k_2$ are disjoint. This is because both of these fields are imaginary quadratic extensions of the same totally real field $k$, and thus these imaginary quadratic extensions must be distinct and not produce any of the same imaginary quadratic subfields. Thus the product $P_1 P_2$ is the product of the class numbers of $2\cdot 2^{n-2}=2^{n-1}$ distinct imaginary quadratic subfields of $K$, which is all of the imaginary quadratic subfields of $K$. Therefore $P=P_1P_2$ is the product of the class numbers of all imaginary quadratic subfields of $K$.

\end{proof}

\begin{rmk} \normalfont Though this theorem presents $Q$ as simply being a positive (rational) integer, we may look back to Kuroda's class number formula in order to glean more information about $Q$. In reality (if you can consider number fields `reality'), $Q$ is a product of indices, which are all positive integers, so we may determine some of the factors of $Q$ to learn more about the class number of a field.
\end{rmk}

We now have all of the tools in place to complete our classification of the imaginary $n$-quadratic fields of class number 1!

\section{A complete classification}\label{main_theorem_section}

We have previously seen complete lists of the imaginary quadratic and biquadratic fields of class number 1. These lists are restated here because it is pleasing to see the complete classification all in one place. We will see in the following theorem that there are several imaginary triquadratic fields of class number 1, and then the list stops abruptly. This is because imaginary $n$-quadratic fields with $n\geq 4$ all have class number larger than 1.

\begin{thm} The imaginary $n$-quadratic fields with class number 1 are:
\begin{enumerate}
\item the nine imaginary quadratic fields $\QQ(\sqrt{a})$ with \newline $a\in\{-1,-2,-3,-7,-11,-19,-43,-67,-163\}$,
\item the 42 imaginary biquadratic fields with radicand lists given in the following table:
 \begin{center}
\begin{tabular}{lllll}
$\{-1,2\}$  & $\{2,-3\}$   & $\{-3,5\}$    & $\{-7,5\}$    & $\{-11,17\}$   \\
$\{-1,3\}$  & $\{2,-11\}$  & $\{-3,-7\}$   & $\{-7,-11\}$  & $\{-11,-19\}$  \\
$\{-1,5\}$  & $\{-2,-3\}$  & $\{-3,-11\}$  & $\{-7,13\}$   & $\{-11,-67\}$  \\
$\{-1,7\}$  & $\{-2,5\}$   & $\{-3,17\}$   & $\{-7,-19\}$  & $\{-11,-163\}$ \\
$\{-1,11\}$ & $\{-2,-7\}$  & $\{-3,-19\}$  & $\{-7,-43\}$  & $\{-19,-67\}$  \\
$\{-1,13\}$ & $\{-2,-11\}$ & $\{-3,41\}$   & $\{-7,61\}$   & $\{-19,-163\}$ \\
$\{-1,19\}$ & $\{-2,-19\}$ & $\{-3,-43\}$  & $\{-7,-163\}$ & $\{-43,-67\}$  \\
$\{-1,37\}$ & $\{-2,29\}$  & $\{-3,-67\}$  &           & $\{-43,-163\}$ \\
$\{-1,43\}$ & $\{-2,-43\}$ & $\{-3,89\}$   &           & $\{-67,-163\}$ \\
$\{-1,67\}$ & $\{-2,-67\}$ & $\{-3,-163\}$ &           &            \\
$\{-1,163\}$ &          &  				 &           &            \\
\end{tabular}, and
\end{center}
\item the 17 imaginary triquadratic fields with the radicand lists:
\begin{center}
\begin{tabular}{lllll}
$\{-1,2,3\}$  & $\{-1,3,5\}$  & $\{-1,7,5\}$  & $\{-2,-3,-7\}$  & $\{-3,-7,5\}$  \\
$\{-1,2,5\}$  & $\{-1,3,7\}$  & $\{-1,7,13\}$ & $\{-2,-3,5\}$ & $\{-3,-11,2\}$  \\
$\{-1,2,11\}$ & $\{-1,3,11\}$ & $\{-1,7,19\}$ & $\{-2,-7,5\}$ & $\{-3,-11,-19\}$ \\
           & $\{-1,3,19\}$ &            &            & $\{-3,-11,17\}$ \\
\end{tabular}.
\end{center}
\end{enumerate}
\end{thm}

We need only to prove (3), and prove that there are no $n$-quadratic fields with $n\geq4$ that have class number 1. The remainder of this paper discusses these results.

\subsection{The imaginary triquadratic fields of class number 1}

For demonstration purposes we will compute the class number of a particular imaginary triquadratic field:

\begin{example}\label{rad123} The number field $K=\QQ(\sqrt{-1},\sqrt{-2},\sqrt{-3})$ has class number 1.
\end{example}

\begin{proof}  Let $k=\QQ(\sqrt{2})$; this is a valid choice for $k$ to use with the formula given in Theorem \ref{smallKuroda} since we know there exists an odd prime $p=3$ which ramifies in $K$ but not in $k$. Let $k_1=\QQ(\sqrt{-1},\sqrt{2})$, $k_2=\QQ(\sqrt{-3},\sqrt{2})$ and $k_3=\QQ(\sqrt{3},\sqrt{2})$. Now $P$ is the product of the class numbers of all imaginary quadratic fields of $K$: $P=h(-1)h(-2)h(-3)h(-6)=1\cdot1\cdot1\cdot2=2$, so
\[h_K=\frac18QPh_3=\frac14Qh_3.\]

To compute the values for $Q$ and $h_3$ we need to understand the unit indices $q(L)$, $L\in\{K,k_1,k_2,k_3\}$. We first find the units of the quadratic subfields. The unit groups $E(\QQ(\sqrt{-2}))$ and $E(\QQ(\sqrt{-6}))$ are simply equal to $\{\pm1\}$. The other unit groups are:
\[\begin{tabular}{rl}
$E(\QQ(\sqrt{-1}))$ & $=\left\{\left(e^{\pi i/2}\right)^\ell:\ell\in\ZZ\right\}$ \\
$E(\QQ(\sqrt{-3}))$ & $=\left\{\left(e^{\pi i/3}\right)^\ell:\ell\in\ZZ\right\}$ \\
$E(\QQ(\sqrt{2}))$ & $=\left\{\pm\left(1+\sqrt{2}\right)^\ell:\ell\in\ZZ\right\}$ \\
$E(\QQ(\sqrt{3}))$&$=\left\{\pm\left(2+\sqrt{3}\right)^\ell:\ell\in\ZZ\right\}$ \\
$E(\QQ(\sqrt{6}))$ &$=\left\{\pm\left(5+2\sqrt{6}\right)^\ell:\ell\in\ZZ\right\}$
\end{tabular}\]

Now, to find $Q=q(K/k)q(k_1/\QQ)q(k_2/\QQ)$ we first note that in~\cite{Kub} Kubota calculates $q(k_1/\QQ)$ and finds that it is equal to 2. To better understand this index, observe that an eighth root of unity, $e^{\pi i/4}=\frac12(\sqrt{2}+\sqrt{-2})$ is in $k_1$ but is not in any of the quadratic subfields of $k_1$. This is what gives rise to the index $q(k_1/\QQ)$ being equal to 2, and in fact tells us that
\[E(k_1)=\left\{\left(e^{\pi i/4}\right)^{\ell_1}\left(1+\sqrt{2}\right)^{\ell_2}:\ell_1,\ell_2\in\ZZ\right\}.\]

Next we must compute $q(k_2/\QQ)$. A fundamental unit of $k_2$ is
\[\frac{\sqrt{2}}{2}-\frac{\sqrt{-6}}{2}+\frac{\sqrt{-3}}{2}-\frac{1}{2} = \left(\sqrt{2}-1\right)\left(\frac{1-\sqrt{-3}}{2}\right)\in E\left(\QQ(\sqrt{2})\right)E\left(\QQ(\sqrt{-3})\right).\]
Also, the roots of unity of $k_2$ are exactly the same as the roots of unity in $\QQ(\sqrt{-3})$. Therefore, 
\[E(k_2)=E\left(\QQ(\sqrt{2})\right)E\left(\QQ(\sqrt{-3})\right)E\left(\QQ(\sqrt{-6})\right).\] 
This gives us $q(k_2/\QQ)=1$.

At this point we have that $Q=q(K/k)q(k_1/\QQ)q(k_2/\QQ)=2q(K/k)$; it is useful to compute $h_3$ before finding $q(K/k)$.

To find $h_3$, recall that Kuroda's class number formula states 
\[h_3=\frac14q(k_3/\QQ)h(2)h(3)h(6)=\frac14q(k_3/\QQ),\]
since the quadratic subfields all have class number 1. Since $k_3$ is a real biquadratic field, we can find the unit group $E(k_3)$, using results from Kubota~\cite{Kub}. To do this, we first take the norm of the fundamental unit of each real quadratic subfield of $k_3$:

\begin{align*}
N_{\QQ(\sqrt{2})/\QQ}(1+\sqrt{2}) &= -1 \\
N_{\QQ(\sqrt{3})/\QQ}(2+\sqrt{3}) &= 1 \\
N_{\QQ(\sqrt{6})/\QQ}(5+2\sqrt{6}) &= 1 \\
\end{align*}

\vspace{-1em}

This tells us that a fundamental system of units for $\cO_{k_3}$ is \[\left\{1+\sqrt{2},\sqrt{2+\sqrt{3}},\sqrt{5+2\sqrt{6}}\right\}.\] Note that we are taking square roots of only those elements of norm 1; this rule does not apply in general, it only applies precisely because there are two fundamental units of norm 1 and one of norm -1. Further, since $k_3$ and all of its subfields are totally real, the only roots of unity contained in these fields are $\pm1$. Putting this together we have that 
\[E(k_3)=\left\{\pm\left(1+\sqrt{2}\right)^{\ell_1}\left(\sqrt{2+\sqrt{3}}\right)^{\ell_2}\left(\sqrt{5+2\sqrt{6}}\right)^{\ell_3}:\ell_1,\ell_2\ell_3\in\ZZ\right\}.\]
Therefore $q(k_3/\QQ)=4$, $h_3=\frac14\cdot4=1$ and
\[h_K=\frac14\cdot 2q(K/k)\cdot 1 =\frac12q(K/k).\]

Since $h_K\in\ZZ$, $q(K/k)$ must be divisible by 2. Recall that 
\[q(K/k)=\left[E(K):E(k_1)E(k_2)E(k_3)\right].\] First note that the roots of unity in $E(K)$ are equal to the roots of unity of $E(k_1)E(k_2)E(k_3)$: both of these groups contain a 24th root of unity. 

By Dirichlet's unit theorem, $E(K)$ has three fundamental units. Computing a fundamental system of units in Sage ~\cite{Sage} we find
\begin{center}
\begin{tabular}{rl}
$\epsilon_1$ & $:=\frac12\sqrt{3} -\frac12\sqrt{-1} - \frac12\sqrt{-3} +\frac12$ \\
$\epsilon_2$ & $:= -\frac14\sqrt{-6}-\frac14\sqrt{2}+\sqrt{-1} +\frac14\sqrt{6} -\frac14\sqrt{-2}$ \\
$\epsilon_3$ & $:=\frac14\sqrt{6}-\frac14\sqrt{2}-\sqrt{-1}+\frac14\sqrt{6}-\frac34\sqrt{-2}+\frac12\sqrt{-3} -\frac12$
\end{tabular}
\end{center}

We can factor the first unit 
\[\epsilon_1=\frac{1-i}{\sqrt{2}}\cdot\frac{1+\sqrt{3}}{\sqrt{2}}=e^{\pi i/4}\cdot\frac{\sqrt{2}+\sqrt{6}}{2}=e^{\pi i/4}\cdot\sqrt{2+\sqrt{3}}.\] 

Finding a factorization for $\epsilon_2$ and $\epsilon_3$ cannot be achieved in Sage. However, dividing one by the other yields:
\begin{align*}\frac{\epsilon_2}{\epsilon_3} &=\frac12\sqrt{-6}-\frac12\sqrt{2}-\frac12\sqrt{-3}+\frac12\\
&=\frac{-1+\sqrt{-3}}{2}\cdot (\sqrt{2}-1) \\
&=e^{2\pi i/3}\cdot (\sqrt{2}-1).
\end{align*}

Thus $\left\{1+\sqrt{2},\sqrt{2+\sqrt{3}},\epsilon_3\right\}$ is a fundamental system of units for $K$. As noted previously, $q(K/k)$ is divisible by 2, so the fact that $1+\sqrt{2},\sqrt{2+\sqrt{3}}$ are units of the subfield $k_3$ of $K$ implies that we must have $\epsilon_3\notin E(k_1)E(k_2)E(k_3)$. Let $E'(K)$ denote the group of units generated by the set
\[\left\{e^{2\pi i/12}, 1+\sqrt{2},\sqrt{2+\sqrt{3}},\epsilon_3^2 \right\}.\]
Then
\[q(K/k)=2\cdot\left[E'(K):E(k_1)E(k_2)E(k_3)\right].\]
Squaring $\epsilon_3$ yields
\begin{align*}
\epsilon_3^2 &= \frac{-1+i}{\sqrt{2}}\cdot\frac{5+4\sqrt{2}-3\sqrt{3}-2\sqrt{6}}{\sqrt{2}} \\
&= e^{3\pi i/4}\cdot\frac12\left(8+5\sqrt{2}-4\sqrt{3}-3\sqrt{6}\right).
\end{align*}
The latter factor is in $E(k_3)$, so the index $\left[E'(K):E(k_1)E(k_2)E(k_3)\right]$ must equal 1. Thus $q(K/k)=2$ and the class number of $K$ is
\[h_K=\frac12q(K/k)=1.\]
\end{proof}

Now let's consider all imaginary triquadratic fields. Looking at Theorem \ref{bigKuroda} it is clear that in order to have $h_K=1$ the product $QP h_3$ must equal $2^{2^{n-1}-1}=2^3$ since $n=3$. This implies that $P$, the product of all imaginary quadratic subfields of $K$, must satisfy
\[P=2^t,\ t\leq 8.\] 

Write a general imaginary triquadratic field $K$ as $K=\QQ(\sqrt{-a_1},\sqrt{-a_2},\sqrt{-a_3})$ with $a_1,a_2,a_3\in\ZZ$ positive and squarefree. Let $a_4$ denote the squarefree part of the product $a_1a_2a_3$, so that the four imaginary quadratic subfields of $K$ are $\QQ(\sqrt{-a_i})$, $1\leq i\leq 4$. Thus we have
\[h(-a_1)h(-a_2)h(-a_3)h(-a_4)=2^t,\ t\leq 8\] 
which implies that $h(a_i)\leq 8$, $1\leq i\leq 4$.

We know all imaginary quadratic fields of class number 1, 2 and 4. The latter two lists are stated below. The first of these, the list of imaginary quadratic fields of class number 2, has a long history and was determined by a number of authors. There is no one paper that determines the entire list, but this determination relied primarily on several papers by Baker and Stark, whose relevant papers are referenced at the beginning of the lemma.

\vspace{0.5em}

\begin{lem} ~\cite{Bak66, Bak71, Sta67, Sta71, Sta72, Sta75} The imaginary quadratic fields of class number 2 are given by $\QQ(\sqrt{-a})$ with\newline 
$a\in\{5,6,10,13,15,22,35,37,51,58,91,115,123,187,235,267,403,427\}$.
\end{lem}

\begin{lem} [ {Arno \cite[Theorem 7]{Arn92}} ]The imaginary quadratic fields of class number 4 are given by $\QQ(\sqrt{-a})$ with\newline 
$a\in\{14,17,21,30,33,34,39,42,46,55,57,70,73,78,82,85,93,97,102,130,133,142,155,177,190,\newline 193,195,203,219,253,259,291,323,355,435,483,555,595,627,667,715,723,763,795,955,1003,\newline 1027,1227,1243,1387,1411,1435,1507,1555\}$.
\end{lem} 

This is enough information to find all possible imaginary quadratic fields $K$ with $P$ equal to a power of 2 not exceeding 8. We will first address the cases where $P\in\{1,2,4\}$. We will consider $P=8$ separately as this case uses slightly different techniques.

\begin{lem}\label{P124} Let $K$ be an imaginary triquadratic field and let $P$ denote the product of the class numbers of the imaginary quadratic subfields of $K$. 
\begin{enumerate}
\item It is impossible to have $P=1$.
\item If $P=2$ then $K$ has radicand list $\{-1,2,3\}$ or $\{-1,2,11\}$.
\item If $P=4$ then $K$ has one of the following radicand lists:
\begin{center}
\begin{tabular}{llll}
$\{ -1 , 2 , 5 \}$ & $\{ -1 , 3 , 11 \}$  & $\{ -1 , 7 , 19 \}$     & $\{ -3 , -7 , -15 \}$ \\
$\{ -1 , 2 , 7 \}$ & $\{-1 , 3 , 19 \}$   & $\{-2 ,-3 ,-7 \}$       & $\{ -3 , -11 , -6 \}$ \\
$\{ -1 , 3 , 7 \}$ & $\{ -1 , 7 , 5 \}$    & $\{ -2 , -3 , -10 \}$  & $\{ -3 , -11 , -19 \}$ \\
$\{- 1 , 3 , 5 \}$ & $\{ -1 , 7 , 13 \}$  &  $\{- 2 ,- 7 , -10 \}$ & $\{ -3 , -11 , -51 \}$
\end{tabular}.
\end{center}
\end{enumerate}
\end{lem}

\begin{proof} We will continue to use the notation used directly above the statement of the lemma. Without loss of generality we can relabel the radicands $-a_i$, $1\leq i\leq4$, in any order, and each radicand will be the squarefree part of the product of the other three. Therefore, we will generally take $h(-a_1)\leq h(-a_2)\leq h(-a_3)\leq h(-a_4)$.

We first prove (1): If $P=1$ then $h(-a_1)=h(-a_2)=h(-a_3)=h(-a_4)=1$. However all quadratic $\QQ(\sqrt{-a})$ with class number 1 and $a>0$ square-free have $a$ either equal to 1 or a prime integer. If $a_1\neq a_2\neq a_3$ are all 1 or prime, then $a_4$, the square-free part of the product $a_1a_2a_3$ will be composite, so $h(-a_4)$ will not equal 1.

Now, consider case (2), where $P=2$. Then we must have $h(-a_1)=h(-a_2)=h(-a_3)=1$ and $h(-a_4)=2$. Examining the lists of imaginary quadratic fields with class numbers 1 and 2 we see that there are only two imaginary triquadratic fields with $h(-a_i)=1$, $1\leq i\leq3$ and $h(-a_4)=2$ are $\QQ(\sqrt{-1},\sqrt{-2},\sqrt{-3})$ (with $a_4=6$) and $\QQ(\sqrt{-1},\sqrt{-2},\sqrt{-11})$ (with $a_4=22$).

In case (3), with $P=4$ we must consider two different scenarios:
\begin{enumerate}[(i)]
\item $h(-a_1)=h(-a_2)=h(-a_3)=1$ and $h(-a_4)=4$ and
\item $h(-a_1)=h(-a_2)=1$, $h(-a_3)=h(-a_4)=2$.
\end{enumerate}
For (i) we look at all combinations satisfying this condition and find 7 possible triquadratic number fields. These number fields have the following values for $a_4$; the values for $a_1,a_2,a_3$ can be recovered by factoring $a_4$:
\[a_4\in\{14, 21, 33, 42, 57, 133, 627\}.\]
Case (ii) yields the remaining radicand lists listed in the statement of the lemma.
\end{proof}

\begin{thm}\label{triquadClassNo1} Of the eighteen imaginary triquadratic fields listed in the statement of Lemma \ref{P124}, seventeen have class number equal to 1. These are the fields with radicand lists given by
\begin{center}
\begin{tabular}{llll}
$\{-1 , 2 , 3\}$ & $\{ -1 , 3 , 11 \}$  & $\{ -1 , 7 , 19 \}$     & $\{ -3 , -7 , -15 \}$ \\
$\{ -1 , 2 , 5 \}$ & $\{-1 , 3 , 19 \}$   & $\{-2 ,-3 ,-7 \}$       & $\{ -3 , -11 , -6 \}$ \\
$\{ -1 , 2 , 11 \}$ & $\{ -1 , 7 , 5 \}$    & $\{ -2 , -3 , -10 \}$  & $\{ -3 , -11 , -19 \}$ \\
$\{ -1 , 3 , 7 \}$ & $\{ -1 , 7 , 13 \}$  &  $\{- 2 ,- 7 , -10 \}$ & $\{ -3 , -11 , -51 \}$ \\
$\{- 1 , 3 , 5 \}$ & & & 
\end{tabular}.
\end{center}
\end{thm}

This proof follows from the same methods in the proof of Example \ref{rad123}. Going through all of the cases of this proof actually requires extensive calculation, and the next section of this paper is devoted to the details of these computations, for the curious reader. We also find that the other field, with radicand list $\{-1,2,7\}$, has class number equal to 2.

\begin{lem} Let $K$ be an imaginary triquadratic field and let $P$ denote the product of the class numbers of the imaginary quadratic subfields of $K$ and $h_K$ denote the class number of $K$. If $P=8$ then $h_K\neq1$.
\end{lem}

\begin{proof}

If $P=8$ then there are three cases to consider:
\begin{enumerate}[(i)]
\item $h(-a_1)=1$ and $h(-a_2)=h(-a_3)=h(-a_4)=2$,\newline
\item $h(-a_1)=h(-a_2)=1$, $h(-a_3)=2$ and $h(-a_4)=4$ or\newline
\item $h(-a_1)=h(-a_2)=h(-a_3)=1$ and $h(-a_4)=8$.\newline
\end{enumerate}

Note that in (i) and (ii) we have $h(-a_1)=1$ and at least one imaginary quadratic subfield of class number 2 or 4. Then any prime that the radicand $a_1$ takes on must divide one of $a_2, a_3$ or $a_4$ (and it can't divide $a_2$ if $h(-a_2)=1$). Some of the radicands appearing in the list for imaginary quadratic fields of class number 1 do not divide any of the radicands of imaginary quadratic fields of class number 2 or 4. These are 43, 67 and 163. Thus, in these cases, the radicand of any imaginary quadratic subfields of class number 1 must be chosen from the set $\{-1,-2,-3,-7,-11,-19\}$. It is also worth noting that 19 does not divide the radicand of any imaginary quadratic field of class number 2.

Now let's look at (i) more closely. We have already observed that $a_1$ must be chosen from the list $\{1,2,3,7,11\}$. Now $-a_2,- a_3,- a_4$ are the radicands of imaginary quadratic fields of class number 2. Again, considering prime factorizations of the radicands, in conjunction with the fact that $a_i=sf\left(\prod_{j\in\{1,2,3,4\}\setminus\{i\}}a_j\right)$ for $1\leq i\leq 4$, we find that $a_2,a_3,a_4$ must be restricted to the set $\{5,6,10,13,15,22,35,51,91,187\}$.

Considering all possible combinations, we find that the radicand lists for the imaginary triquadratic fields $K$ satisfying case (i) are given by
\begin{center}
\begin{tabular}{ll}
$\{-1,-6,-10\}$ (with $a_4=15$) & $\{-2,-5,-6\}$ (with $a_4=15$) \\
$\{-3,-5,-6\}$ (with $a_4=10$) & $\{-11,-5,-10\}$ (with $a_4=22$) 
\end{tabular}.
\end{center}

In case (ii) recall that $a_1,a_2\in\{1,2,3,7,11,19\}$.  Since we have complete lists of all imaginary quadratic fields of class numbers 2 and 4 we can write a simple program to find all imaginary triquadratic fields satisfying (ii). This yields
\begin{center}
\begin{tabular}{ll}
$\{ 1 , 2 , 15 \}$  (with $a_4= 30 $) & $\{ 1 , 2 , 35 \}$  (with $a_4= 70 $) \\
$\{ 1 , 2 , 51 \}$  (with $a_4= 102 $) & $\{ 1 , 3 , 10 \}$  (with $a_4= 30 $) \\
$\{ 1 , 3 , 13 \}$  (with $a_4= 39 $) & $\{ 1 , 7 , 6 \}$  (with $a_4= 42 $) \\
$\{ 1 , 7 , 10 \}$  (with $a_4= 70 $) & $\{ 1 , 7 , 37 \}$  (with $a_4= 259 $) \\
$\{ 1 , 11 , 5 \}$  (with $a_4= 55 $) & $\{ 1 , 19 , 10 \}$  (with $a_4= 190 $) \\
$\{ 2 , 3 , 5 \}$  (with $a_4= 30 $) & $\{ 2 , 3 , 13 \}$  (with $a_4= 78 $) \\
$\{ 2 , 7 , 5 \}$  (with $a_4= 70 $) & $\{ 2 , 19 , 5 \}$  (with $a_4= 190 $) \\
\end{tabular}.
\end{center}

Last consider (iii). Since we do not have a complete list of the imaginary quadratic fields of class number 8, we must tackle this case by considering all possibilities with $h(-a_i)=1$, $1\leq i \leq 3$, and checking if the resulting value for $a_4$ satisfies $h(-a_4)=8$. In this case we check for $h(-a_4)=8$ computationally using Sage. This yields 8 imaginary triquadratic fields $K$ satisfying (iii):
\begin{center}
\begin{tabular}{ll}
$\{-2,-3,-11\}$ (with $a_4=66$) & $\{-2,-3,-43\}$ (with $a_4=258$) \\
$\{-1,-7,-11\}$ (with $a_4=77$) & $\{-1,-7,-43\}$ (with $a_4=301$) \\
$\{-2,-3,-19\}$ (with $a_4=114$) & $\{-2,-11,-19\}$ (with $a_4=418$) \\
$\{-2,-7,-11\}$ (with $a_4=154$) & $\{-3,-19,-43\}$ (with $a_4=2451$) 
\end{tabular}.
\end{center}

Now we must check whether or not any of these fields have class number 1. Since $P=8$ then for an appropriate choice of a real quadratic subfield of $K$ in each case we have that
\[h_K=\frac18Pq(K/k)q(k_1/\QQ)q(k_2/\QQ)h_3=q(K/k)q(k_1/\QQ)q(k_2/\QQ)h_3.\]
Thus if, for an appropriate choice of $k$, any of the factors on the right hand side of this equation are greater than 1, we have that $h_K>1$. It turns out that this is always the case, yielding no more imaginary triquadratic fields of class number 1.

\end{proof}

\subsection{Class Numbers of $n$-quadratic Fields, $n\geq4$}

In this section we will prove that all imaginary $n$-quadratic fields with $n\geq4$ have a non-trivial ideal class group.

First, recall from Theorem \ref{ramPrimes} that if at least 5 rational primes ramify in $K$ then $2\mid h_K$. Also, Lemma \ref{numberOfRamifiedPrimes} states that an imaginary $n$-quadratic field has at least $n-1$ ramified primes. From this we can immediately conclude that any $n$-quadratic field with $n\geq6$ does not have class number 1. Thus, in the remainder of this section we explore the cases of $n=4$ and $n=5$. 

\begin{lem} If $K$ is an imaginary $n$-quadratic field with $n=5$ then $K$ has class number greater than 1. \end{lem}

\begin{proof} Assume, for sake of contradiction, that $K$ is a 5-quadratic field such that $h_K=1$. From Theorem \ref{ramPrimes} we know $K$ must have fewer than 5 ramified primes, so its radicand list must be of the form $\{-1,-2,-p_1,-p_2,-p_3\}$ for $p_1,p_2,p_3\in\ZZ$ prime. Note that $-2$ must be included in this radicand list because the prime $(2)$ ramifies in $\QQ(\sqrt{-1})$.

Recall the formula given in Theorem \ref{bigKuroda}:
\[h_K=\left(\frac12\right)^{2^{5-1}-1}QP h_3=\left(\frac12\right)^{15}QPh_3\]
where $P$ is the product of the class numbers of all of the imaginary quadratic subfields of $K$.

If $h_K=1$ then this implies all imaginary quadratic subfields of $K$ must have class number which is a power of 2, and this product cannot exceed $2^{15}$. In order to bound $P$, we can consider the number of prime factors of the radicands of each of the sixteen imaginary quadratic subfields of $K$. First, we know that $h(-1)=1$. Of the remaining fields:
\begin{enumerate}[i. ]
\item 4 have prime radicands,
\item $\binom{4}{2}=6$ have radicands which are the product of two primes and
\item the remaining 5 imaginary quadratic subfields have radicands which are the product of at least three primes.
\end{enumerate}

We know that if an imaginary quadratic field has class number equal to 1 its radicand is either $-1$ or prime. We can also observe that if an imaginary quadratic field has class number 2 its radicand is the product of at most two primes. Using this information we can bound $P$:
\[P\geq1\cdot1^4\cdot2^6\cdot4^5=2^{16}.\]
Thus we cannot have $h_K=1$.
\end{proof}

\begin{lem} If $K$ is an imaginary $n$-quadratic field with $n=4$ then $K$ has class number greater than 1. \end{lem}

\begin{proof} Theorem \ref{bigKuroda} implies that if $h_K=1$ we have
\[1=h_K\geq \left(\frac12\right)^{7}P.\]
We know that $K$ has $2^{4-1}=8$ imaginary quadratic subfields. If more than 2 of these subfields had class number larger than 4, we would have
\[1=h_K\geq\frac{8^3}{2^7}>1,\]
which would give a contradiction. Therefore, at least 6 of the 8 imaginary quadratic subfields of $K$ must have class number 1, 2 or 4; call these fields $\QQ(\sqrt{-a_i})$, $1\leq i\leq 6$; and write $\QQ(\sqrt{-a_7}),\QQ(\sqrt{-a_8})$ for the remaining two imaginary quadratic subfields.\newline
Then clearly $\QQ(\sqrt{-a_1a_2a_3}), \QQ(\sqrt{-a_1a_2a_4}), \QQ(\sqrt{-a_1a_2a_5})\subset K$. Since the $a_i$ are distinct, at most 2 of these fields may be equal to $\QQ(\sqrt{-a_7})$ or $\QQ(\sqrt{-a_8})$. Thus, without loss of generality, assume $\QQ(\sqrt{-a_1a_2a_3})\neq$ $\QQ(\sqrt{-a_7})$ or $\QQ(\sqrt{-a_8})$. Then we may define \[k_1 = \QQ(\sqrt{-a_1},\sqrt{-a_2},\sqrt{-a_3});\]
so $k_1$ is an imaginary triquadratic field with all imginary quadratic subfields having class number 1, 2 or 4. Further, at most one of $\sqrt{-a_4}$, $\sqrt{-a_5}$ may be contained in $k_1$; assume without loss of generality that $\sqrt{-a_4}\notin k_1$. Then $K=k_1(\sqrt{-a_4})$.

Therefore, a list of candidates for imaginary 4-quadratic fields $K$ with class number 1 may be constructed as follows:
\begin{enumerate}
\item Make a list $S$ of all imaginary triquadratic fields $k_1$ such that all imaginary quadratic subfields have class number 1, 2 or 4.
\item Make a second list, $T$, as follows: for each $k_1\in S$ and each $a$ such that $\sqrt{-a}\notin k_1$ and $\QQ(\sqrt{-a})$ has class number 1, 2 or 4, add $k(\sqrt{-a})$ to $T$. Remove all repeat entries from this list.
\item For each $K\in T$, we want the product $P$ of the class numbers of the imaginary quadratic subfields to not exceed $2^7$. Find a lower bound for $P$ as follows: start with $P=1$ and iterate through all imaginary quadratic subfields; if we know the class number of that subfield, multiply $P$ by that number, otherwise, multiply $P$ by 8. If the lower bound for $P$ is greater than $2^7$, discard that field from the possibility list.
\end{enumerate}

Following these steps yields only four candidates for imaginary 4-quadratic fields with class number 1: $\QQ(\sqrt{-1},\sqrt{-2},\sqrt{-3},\sqrt{-7})$ (has class number 4), $\QQ(\sqrt{-1},\sqrt{-2},\sqrt{-3},\sqrt{-11})$ (has class number 4),\newline $\QQ(\sqrt{-1},\sqrt{-2},\sqrt{-3},\sqrt{-5})$ (has class number 2), and $\QQ(\sqrt{-1},\sqrt{-2},\sqrt{-5},\sqrt{-7})$ (has class number 4).
\end{proof}

Thus we have completed the classification of imaginary $n$-quadratic fields of class number 1!

\section{Computations for Theorem \ref{triquadClassNo1}}

In this section, we look at some of the details for the computation of the class numbers of the list given in Theorem \ref{triquadClassNo1}. This is done by consider the number fields stated in Lemma \ref{P124}, and giving the readers the tools to reproduce a proof of the class number of each field, such as the proof in Example \ref{rad123}. Since the example computes the class number of field with radicand list $\{-1,-2,-3\}$ it only remains to compute the class numbers for the following imaginary triquadratic fields:

\begin{center}
\begin{tabular}{llll}
$\{-1 , 2 , 5\}$ & $\{ -1 , 3 , 11 \}$  & $\{ -1 , 7 , 19 \}$     & $\{ -3 , -7 , -15 \}$ \\
$\{ -1 , 2 , 7 \}$ & $\{-1 , 3 , 19 \}$   & $\{-2 ,-3 ,-7 \}$       & $\{ -3 , -11 , -6 \}$ \\
$\{ -1 , 2 , 11 \}$ & $\{ -1 , 7 , 5 \}$    & $\{ -2 , -3 , -10 \}$  & $\{ -3 , -11 , -19 \}$ \\
$\{ -1 , 3 , 7 \}$ & $\{ -1 , 7 , 13 \}$  &  $\{- 2 ,- 7 , -10 \}$ & $\{ -3 , -11 , -51 \}$ \\
$\{- 1 , 3 , 5 \}$ & & & 
\end{tabular}.
\end{center}

Now, to compute the class numbers for each of these imaginary triquadratic fields $K$ we use a version of Kuroda's class number formula: \[h_{K}=\frac18 Q P h_{3}.\] Here $Q$ is the product of unit indices, \[Q=q(K/k)q(k_{1}/\QQ)q(k_{2}/\QQ),\] $P$ is the product of the class numbers of all imaginary quadratic subfields of $K$ and $h_{3}$ is the class number of a real biquadratic subfield of $K$ (which will be specified later on in this section).

Now, to each triquadratic field $K$ that we are computing the class number for, we must associate to it an appropriate choice of subfields $k_{1}$, $k_{2}$, $k_{3}$ and $k$. Recall that $k_{3}$ and $k$ must be totally real, and that there must be a rational prime $p$ that ramifies in $K$ but not in $k$. Referring to each field by its radicand list, the following table shows and appropriate choice for these values, and specifies a prime that displays this ramification property. Note that these choices of subfields could, in most cases, be chosen differently and still be useful for computing class numbers.

As there are 17 fields in question, it is helpful to index them so we can refer to the different fields throughout this section. 

\begin{center}
\begin{tabular}{l|llllll}
$i$	& $K_i$ 				& $k_i$ 	& $k_{1,i}$ 	& $k_{2,i}$ 	& $k_{3,i}$ 	& $p_i$ \\
\hline
1	& $\{ -1 , 2 , 11 \}$ 		& $\{2\}$ 	& $\{-1 , 2\}$	& $\{-11 , 2\}$	& $\{11 , 2\}$	& 11	\\
2	& $\{-1 , 2 , 5\}$ 			& $\{2\}$ 	& $\{-1 , 2\}$	& $\{-5 , 2\}$	& $\{5 , 2\}$	& 5	\\
3	& $\{ -1 , 2 , 7 \}$ 		& $\{2\}$ 	& $\{-1 , 2\}$	& $\{-7 , 2\}$	& $\{7 , 2\}$	& 7	\\
4	& $\{- 1 , 3 , 5 \}$ 		& $\{3\}$	& $\{-1 , 3\}$	& $\{-5 , 3\}$	& $\{5 , 3\}$	& 5	\\
5	& $\{ -1 , 3 , 7 \}$ 		& $\{3\}$	& $\{-1 , 3\}$	& $\{-7 , 3\}$	& $\{7 , 3\}$	& 7	\\
6	& $\{ -1 , 3 , 11 \}$		& $\{3\}$	& $\{-1 , 3\}$	& $\{-11 , 3\}$	& $\{11 , 3\}$	& 11	\\
7	& $\{-1 , 3 , 19 \}$		& $\{3\}$	& $\{-1 , 3\}$	& $\{-19 , 3\}$	& $\{19 , 3\}$	& 19	\\
8	& $\{ -1 , 7 , 5 \}$		& $\{7\}$	& $\{-1 , 7\}$	& $\{-5 , 7\}$	& $\{5 , 7\}$	& 5	\\
9	& $\{ -1 , 7 , 13 \}$		& $\{7\}$	& $\{-1 , 7\}$	& $\{-13 , 7\}$	& $\{13 , 7\}$	&13	\\
10	& $\{ -1 , 7 , 19 \}$		& $\{7\}$	& $\{-1 , 7\}$	& $\{-19 , 7\}$	& $\{19 , 7\}$	&19	\\
11	& $\{-2 ,-3 ,-7 \}$			& $\{6\}$	& $\{-2, 6\}$	& $\{-7,6\}$	& $\{14, 6\}$	& 7	\\
12	& $\{ -2 , -3 , -10 \}$		& $\{5\}$	& $\{-2, 5\}$	& $\{-3, 5\}$	& $\{6, 5\}$	& 3	\\
13	& $\{- 2 ,- 7 , -10 \}$		& $\{5\}$	& $\{-2, 5\}$	& $\{-7, 5\}$	& $\{14, 5\}$	& 7	\\
14	& $\{ -3 , -7 , -15 \}$ 		& $\{5\}$	& $\{-3, 5\}$	& $\{-7, 5\}$	& $\{21, 5\}$	& 7	\\
15	& $\{ -3 , -11 , -6 \}$ 		& $\{2\}$	& $\{-3, 2\}$	& $\{-11, 2\}$	& $\{33, 2\}$	& 3	\\
16	& $\{ -3 , -11 , -19 \}$ 		& $\{33\}$	& $\{-3, 33\}$	& $\{-19, 33\}$	& $\{57, 33\}$	& 19	\\
17	& $\{ -3 , -11 , 17 \}$ 		& $\{33\}$	& $\{-3, 33\}$	& $\{-51, 33\}$	& $\{17, 33\}$	& 17	\\
\end{tabular}
\end{center}

Now we must compute the corresponding values $Q$ and $h_3$ for each of these fields $K_i$. As these values all deal with the unit indices, we need to know the unit groups of each of the quadratic subfields of these triquadratic fields. We begin with the imaginary quadratic fields, as their unit groups are finite and very straightforward:

\renewcommand{\arraystretch}{1.5}

\begin{center}
\begin{tabular}{lll}
Imaginary Quadratic Field  & \hspace{0.3in} & Unit Group \\
\hline 
$\QQ\left(\sqrt{-1}\right)$ 		&& 	$\left\{\left(e^{\pi i/2}\right)^\ell:\ell\in\ZZ\right\}$ \\
$\QQ\left(\sqrt{-3}\right)$ 		&&	$\left\{\left(e^{\pi i/3}\right)^\ell:\ell\in\ZZ\right\}$ \\
$\QQ\left(\sqrt{-a}\right),\ a\in\NN, a\neq1,3, a\text{ squarefree}$ 		&&	$\{\pm1\}$ \\
\end{tabular}
\end{center}

\vspace{0.5em}

We now must compute the unit group of all of the real quadratic fields that are relevant to this proof. Here we list a fundamental unit $\epsilon$ for each field, calculated using Sage. If $\epsilon$ is a fundamental unit of some real quadratic field $\QQ(\sqrt{a})$ then the unit group $E(\QQ(\sqrt{a}))$ is equal to $\{\pm\epsilon^\ell : \ell\in\ZZ\}$.

Additionally, this list contains the norm of each fundamental unit, as this is also required in the computations of the unit indices.

\begin{center}
\begin{tabular}{ccr}
Real Quadratic Field  & Fundamental Unit, $\epsilon$ 		& $N(\epsilon)$ \\
\hline 
$\QQ\left(\sqrt{2}\right)$	& $1+\sqrt{2}$				& $-1$ \\
$\QQ\left(\sqrt{3}\right)$	& $2+\sqrt{3}$				& 1 \\	
$\QQ\left(\sqrt{5}\right)$ 	& $\frac12\left(1+\sqrt{5}\right)$ 	& $-1$ \\
$\QQ\left(\sqrt{6}\right)$  	& $5+2\sqrt{6}$				& 1 \\	
$\QQ\left(\sqrt{7}\right)$ 	& $8+3\sqrt{7}$				& 1 \\	
$\QQ\left(\sqrt{10}\right)$ 	& $3+\sqrt{10}$				& $-1$ \\
$\QQ\left(\sqrt{11}\right)$  	& $10+3\sqrt{11}$			& 1 \\	
$\QQ\left(\sqrt{14}\right)$  	& $15+4\sqrt{14}$			& 1 \\	
$\QQ\left(\sqrt{15}\right)$  	& $4+\sqrt{15}$				& 1 \\	
$\QQ\left(\sqrt{17}\right)$  	& $4+\sqrt{17}$				& $-1$ \\	
$\QQ\left(\sqrt{19}\right)$  	& $170+39\sqrt{19}$			& 1 \\	
$\QQ\left(\sqrt{21}\right)$ 	& $\frac12\left(5+\sqrt{21}\right)$ & 1 \\
$\QQ\left(\sqrt{22}\right)$	& $197+42\sqrt{22}$ 			& 1 \\
$\QQ\left(\sqrt{30}\right)$ 	& $11+2\sqrt{30}$ 			& 1 \\
$\QQ\left(\sqrt{33}\right)$	& $23+4\sqrt{33}$			& 1 \\
$\QQ\left(\sqrt{35}\right)$	& $6+\sqrt{35}$ 				& 1 \\
$\QQ\left(\sqrt{57}\right)$	& $151+20\sqrt{57}$ 			& 1 \\
$\QQ\left(\sqrt{66}\right)$	& $65+8\sqrt{66}$ 			& 1 \\
$\QQ\left(\sqrt{70}\right)$	& $251+30\sqrt{70}$ 			& 1 \\	
$\QQ\left(\sqrt{91}\right)$	& $1574+165\sqrt{91}$ 		& 1 \\
$\QQ\left(\sqrt{105}\right)$	& $41+4\sqrt{105}$ 			& 1 \\
$\QQ\left(\sqrt{209}\right)$	& $46551+3220\sqrt{209}$ 		& 1 \\
\end{tabular}
\end{center}

\vspace{0.5em}

Using the fundamental units and their norms, we can find the class number of each $k_{3,i}$, $1\leq i\leq 17$ using the results from Kubota's paper ~\cite{Kub}; this paper presents many examples of computing these class numbers and covers all of the cases we are concerned with here. These class numbers can also be checked independently in Sage. Using either process to calculate the class numbers of these real biquadratic subfields, we find that the class numbers of all of the $k_{3,i}$ are equal to 1. 

Also, recall that Lemma \ref{P124} gives us the values of $P$ for each $K$. When $K=\QQ\left(\sqrt{-1},\sqrt{2},\sqrt{11}\right)$ we have that $P=2$ and, for all other triquadratic fields in the above list, $P=4$. Putting the information from this paragraph together with the last paragraph, we can reduce our formulae for the class numbers of the imaginary triquadratic fields:
\[h_{K_1}=\frac14Q,\text{ and } h_{K_i}=\frac12Q\text{ for }2\leq i\leq 17.\]

Thus all that remains is computing the values of $Q$, which (unfortunately) is the most computationally heavy portion of this process. This requires computing the unit group of each of our triquadratic fields $K_i$ and the biquadratic fields $k_{1,i}$, $k_{2,i}$ and $k_{1,i}$.

\begin{center}
\begin{tabular}{l|llllll}
$i$	& $K_i$ 				& $k_i$ 	& $k_{1,i}$ 	& $k_{2,i}$ 	& $k_{3,i}$ 	& $p_i$ \\
\hline
1	& $\{ -1 , 2 , 11 \}$ 		& $\{2\}$ 	& $\{-1 , 2\}$	& $\{-11 , 2\}$	& $\{11 , 2\}$	& 11	\\
2	& $\{-1 , 2 , 5\}$ 			& $\{2\}$ 	& $\{-1 , 2\}$	& $\{-5 , 2\}$	& $\{5 , 2\}$	& 5	\\
3	& $\{ -1 , 2 , 7 \}$ 		& $\{2\}$ 	& $\{-1 , 2\}$	& $\{-7 , 2\}$	& $\{7 , 2\}$	& 7	\\
4	& $\{- 1 , 3 , 5 \}$ 		& $\{3\}$	& $\{-1 , 3\}$	& $\{-5 , 3\}$	& $\{5 , 3\}$	& 5	\\
5	& $\{ -1 , 3 , 7 \}$ 		& $\{3\}$	& $\{-1 , 3\}$	& $\{-7 , 3\}$	& $\{7 , 3\}$	& 7	\\
6	& $\{ -1 , 3 , 11 \}$		& $\{3\}$	& $\{-1 , 3\}$	& $\{-11 , 3\}$	& $\{11 , 3\}$	& 11	\\
7	& $\{-1 , 3 , 19 \}$		& $\{3\}$	& $\{-1 , 3\}$	& $\{-19 , 3\}$	& $\{19 , 3\}$	& 19	\\
8	& $\{ -1 , 7 , 5 \}$		& $\{7\}$	& $\{-1 , 7\}$	& $\{-5 , 7\}$	& $\{5 , 7\}$	& 5	\\
9	& $\{ -1 , 7 , 13 \}$		& $\{7\}$	& $\{-1 , 7\}$	& $\{-13 , 7\}$	& $\{13 , 7\}$	&13	\\
10	& $\{ -1 , 7 , 19 \}$		& $\{7\}$	& $\{-1 , 7\}$	& $\{-19 , 7\}$	& $\{19 , 7\}$	&19	\\
11	& $\{-2 ,-3 ,-7 \}$			& $\{6\}$	& $\{-2, 6\}$	& $\{-7,6\}$	& $\{14, 6\}$	& 7	\\
12	& $\{ -2 , -3 , -10 \}$		& $\{5\}$	& $\{-2, 5\}$	& $\{-3, 5\}$	& $\{6, 5\}$	& 3	\\
13	& $\{- 2 ,- 7 , -10 \}$		& $\{5\}$	& $\{-2, 5\}$	& $\{-7, 5\}$	& $\{14, 5\}$	& 7	\\
14	& $\{ -3 , -7 , -15 \}$ 		& $\{5\}$	& $\{-3, 5\}$	& $\{-7, 5\}$	& $\{21, 5\}$	& 7	\\
15	& $\{ -3 , -11 , -6 \}$ 		& $\{2\}$	& $\{-3, 2\}$	& $\{-11, 2\}$	& $\{33, 2\}$	& 3	\\
16	& $\{ -3 , -11 , -19 \}$ 		& $\{33\}$	& $\{-3, 33\}$	& $\{-19, 33\}$	& $\{57, 33\}$	& 19	\\
17	& $\{ -3 , -11 , 17 \}$ 		& $\{33\}$	& $\{-3, 33\}$	& $\{-51, 33\}$	& $\{17, 33\}$	& 17	\\
\end{tabular}
\end{center}

We can make use of the unit indices computed in Example \ref{rad123}: $q(\QQ(\sqrt{-1},\sqrt{-2})/\QQ)=2$ and $q(\QQ(\sqrt{-3},\sqrt{-2})/\QQ)=1$.

\end{document}